\documentclass[12pt]{amsart}

\usepackage{ucs}

\usepackage{amssymb}
\usepackage{amsthm}
\usepackage{amsmath}
\usepackage{latexsym}
\usepackage[cp1251]{inputenc}
\usepackage{graphicx}
\usepackage{wrapfig}
\usepackage{caption}
\usepackage{subcaption}
\usepackage{indentfirst}
\usepackage[left=2.6cm,right=2.6cm,top=2.6cm,bottom=2.6cm,bindingoffset=0cm]{geometry}
\usepackage{enumerate}
\usepackage{makecell}

\DeclareMathOperator{\aut}{Aut}

\DeclareMathOperator{\cyc}{Cyc}

\DeclareMathOperator{\orb}{Orb}

\DeclareMathOperator{\Span}{Span}

\DeclareMathOperator{\sym}{Sym}

\makeatletter 
\def\@seccntformat#1{\csname the#1\endcsname. } 
\def\@biblabel#1{#1.}

\makeatother

\title{On nilpotent Schur groups}

\author{Grigory Ryabov}
\address{Sobolev Institute of Mathematics, Novosibirsk, Russia}
\address{Novosibirsk State Technical University, Novosibirsk, Russia}

\email{gric2ryabov@gmail.com}

\thanks{The work is supported by Russian Scientific Fund (project No.~22-71-00021)}

\date{}

\newtheorem{prop}{Proposition}[section]

\newtheorem{lemm}[prop]{Lemma}
\newtheorem{theo}[prop]{Theorem}
\newtheorem*{ques}{Question}

\theoremstyle{definition}

\newtheorem*{prob}{Problem}

\begin{document}

\vspace{\baselineskip}
\vspace{\baselineskip}

\vspace{\baselineskip}

\vspace{\baselineskip}

\begin{abstract}
A finite group $G$ is called a \emph{Schur} group if every $S$-ring over $G$ is \emph{schurian}, i.e. associated in a natural way with a subgroup of $\sym(G)$ that contains all right translations. We prove that every nonabelian nilpotent Schur group belongs to one of the explicitly given families of groups. 
\end{abstract}

\maketitle

\section{Introduction}

A \emph{Schur ring} or \emph{$S$-ring} over a finite group $G$ can be defined as a subring of the group ring $\mathbb{Z}G$ that is a free $\mathbb{Z}$-module spanned by a partition of $G$ closed under taking inverse and containing the identity element $e$ of $G$ as a class (see Section~2 for the exact definition). The theory of $S$-rings was initiated by Schur~\cite{Schur}, and further it was developed by Wielandt~\cite{Wi}. Schur and Wielandt used $S$-rings for studying permutation groups. In particular, Schur proved a generalization of the Burnside theorem on a permutation group with a regular cyclic $p$-subgroup using $S$-rings~\cite[Theorem~25.3]{Wi}.

With every group $K\leq \sym(G)$ containing all right translations of~$G$, one can associate the $S$-ring determined by the partition of $G$ into the orbits of the stabilizer of~$e$ in~$K$. Every $S$-ring obtained in such way was called \emph{schurian} in~\cite{Po}. Wielandt wrote in~\cite{Wi2} that ``Schur had conjectured for a long time that every $S$-ring is determined by a suitable permutation group'', in our terms that every $S$-ring is schurian. However, this conjecture was disproved in~\cite{Wi} by Wielandt. In~\cite{Po}, P\"{o}schel introduced the following definition. The group $G$ is called a \emph{Schur} group if every $S$-ring over $G$ is schurian. He also posed the next problem.

\begin{prob}
\emph{Determine all Schur groups.} 
\end{prob}

To prove that a given group $G$ is Schur, it is necessary to check that every $S$-ring over $G$ is schurian. The main difficulty of the above problem is that the number of $S$-rings over $G$ is exponential in the order of $G$ in general, and sometimes it is not easy to obtain a description of all $S$-rings over~$G$. On the other hand, to prove that $G$ is not Schur, it is necessary to find a nonschurian $S$-ring over $G$. However, in many cases it is unclear how to find such $S$-ring. 

The first result on the schurity problem was obtained by  P\"{o}schel~\cite{Po}. He proved that cyclic $p$-groups of odd order are Schur. Using this result, Klin and P\"{o}schel solved the isomorphism for Cayley graphs over these groups~\cite{KP}. Schurity of cyclic $2$-groups was proved in~\cite{GNP}. The complete classification of cyclic Schur groups was obtained in~\cite{EKP1}. 

Strong necessary conditions of schurity for abelian groups were obtained in~\cite{EKP2}. Schurity of some infinite families of noncyclic abelian groups was verified in~\cite{MP2,Ry2}. Abelian Schur groups of odd order were classified in~\cite{PR}. Finally, the complete classification of abelian Schur groups was obtained in~\cite{Ry3}.

All Schur groups of order at most~$63$ were enumerated in~\cite{Ziv}. It turns out that there exist nonabelian Schur groups. However, it is not known to the present moment whether there exists an infinite family of nonabelian Schur groups. The attempts to describe all $S$-rings over some nonabelian groups, e.g. dihedral groups or groups whose order has few prime divisors, lead to such hard unsolved problems as the problem of existence of a difference set in a cyclic group and the problem of computing cyclotomic numbers (see~\cite{MP2}).

Nevertheless, there are some general results on nonabelian Schur groups. In~\cite{PV}, it was proved that every nonabelian Schur group $G$ is metabelian and the number of distinct prime divisors of $|G|$ is at most~$7$. Some partial results one nonabelian Schur groups were obtained in~\cite{MP2,Ry1}. In particular, in~\cite{Ry1} it was proved that every Schur $p$-group of odd order must be abelian and in~\cite{MP2} it was proved that every nonabelian $2$-group of order at least~$32$ must be dihedral.

In this paper, we prove that every nilpotent Schur group belongs to one of the explicitly given families of groups. The cyclic and dihedral groups of order~$n$ are denoted by $\mathbb{Z}_n$ and $D_n$, respectively. The quaternion group and the central product of $D_8$ and $\mathbb{Z}_4$ are denoted by $Q_8$ and $G_{16}$, respectively. The main result of the paper is the following theorem.

\begin{theo}\label{main}
A nonabelian nilpotent Schur group is isomorphic to one of the groups below:
\\
\\
$(1)$ $Q_8$, $G_{16}$, $D_{2^k}$, where $k\geq 3$,
\\
\\
$(2)$ $Q_8\times \mathbb{Z}_p$, where $p\geq 11$ is a prime such that $p\not\equiv1\mod 4$ and $p\not\equiv1\mod 6$.
\\
\\
Moreover, the groups $Q_8$, $G_{16}$, and $D_{2^k}$, where $3\leq k\leq 5$, are Schur.  
\end{theo}

We do not know whether infinite families of groups from Theorem~\ref{main} consist of Schur groups. So we suggest the following question.

\begin{ques}
Are the groups $D_{2^k}$, $k\geq 6$, and $Q_8\times \mathbb{Z}_p$, $p\not\equiv1\mod 4$, $p\not\equiv1\mod 6$, Schur?
\end{ques}

\noindent It should be mentioned that all $S$-rings of rank at most~$5$ over $D_{2^k}$, $k\geq 3$, were classified in~\cite{MP2}.

We finish the introduction with a brief outline of the paper. Section~$2$ contains a necessary background on $S$-rings. In Sections~$3$ and~$4$, we construct new infinite families of nonschurian $S$-rings over the groups $D_8\times \mathbb{Z}_p$ and $Q_8\times \mathbb{Z}_p$, where $p$ is prime; in case of $Q_8\times \mathbb{Z}_p$, we also assume that $p\equiv1\mod 4$ or $p\equiv1\mod 6$, respectively. The latter condition on~$p$ is substantial for constructing a nonschurian $S$-ring over $Q_8\times \mathbb{Z}_p$. Finally, we prove Theorem~\ref{main} in Section~$5$. One of the key ingredients of the proof is~\cite[Corollary~4.3]{PV}, which provides some necessary conditions of schurity for nilpotent groups.

The author would like to thank Dr. M. Ziv-Av for the help with computer calculations.

\section{Preliminaries}

In this section, we provide all necessary definitions and statements on $S$-rings. All of them can be found, e.g., in~\cite{MP2,Ry2}.

Let $G$ be a finite group and $\mathbb{Z}G$  the integer group ring. The identity element and the set of all nonidentity elements of $G$ are denoted by~$e$ and~$G^\#$, respectively. The symmetric group of the set $G$ is denoted by~$\sym(G)$. If $K\leq \sym(G)$, then the set of all orbits of $K$ on $G$ is denoted by $\orb(K,G)$. The subgroup of $\sym(G)$ induced by the right multiplications of $G$ is denoted by $G_{r}$. If $X\subseteq G$, then the element $\sum \limits_{x\in X} {x}$ of the group ring $\mathbb{Z}G$ is denoted by~$\underline{X}$. The set $\{x^{-1}:x\in X\}$ is denoted by $X^{-1}$.

A subring  $\mathcal{A}\subseteq \mathbb{Z} G$ is called an \emph{$S$-ring} (a \emph{Schur} ring) over $G$ if there exists a partition $\mathcal{S}=\mathcal{S}(\mathcal{A})$ of~$G$ such that:

$(1)$ $\{e\}\in\mathcal{S}$;

$(2)$  if $X\in\mathcal{S}$, then $X^{-1}\in\mathcal{S}$;

$(3)$ $\mathcal{A}=\Span_{\mathbb{Z}}\{\underline{X}:\ X\in\mathcal{S}\}$.

\noindent The elements of $\mathcal{S}$ are called the \emph{basic sets} of  $\mathcal{A}$. A set $X \subseteq G$ is called an \emph{$\mathcal{A}$-set} if $\underline{X}\in \mathcal{A}$. A subgroup $A \leq G$ is called an \emph{$\mathcal{A}$-subgroup} if $A$ is an $\mathcal{A}$-set. 

Let $\{e\}\leq A \unlhd B\leq G$. A section $B/A$ is called an \emph{$\mathcal{A}$-section} if $B$ and $A$ are $\mathcal{A}$-subgroups. If $S=B/A$ is an $\mathcal{A}$-section, then the module
$$\mathcal{A}_S=Span_{\mathbb{Z}}\left\{\underline{X}^{\pi}:~X\in\mathcal{S}(\mathcal{A}),~X\subseteq B\right\},$$
where $\pi:B\rightarrow B/A$ is the canonical epimorphism, is an $S$-ring over $S$.

Let $S=B/A$ be an $\mathcal{A}$-section of $G$. The $S$-ring~$\mathcal{A}$ is called the \emph{$S$-wreath product} of $\mathcal{A}_B$ and $\mathcal{A}_{G/A}$ if $A\trianglelefteq G$ and every basic set $X$ of $\mathcal{A}$ outside~$B$ is a union of some $A$-cosets. The $S$-wreath product is called \emph{nontrivial} or \emph{proper}  if $A\neq \{e\}$ and $B\neq G$. 

Let $A$ and $B$ be $\mathcal{A}$-subgroups such that $G=A\times B$. The $S$-ring~$\mathcal{A}$ is called the \emph{tensor product} of $\mathcal{A}_A$ and $\mathcal{A}_B$ if 
$$\mathcal{S}(\mathcal{A})=\{X_1\times X_2:~X_1\in\mathcal{S}(\mathcal{A}_A),~X_2\in \mathcal{S}(\mathcal{A}_B)\}.$$
In this case we write $\mathcal{A}=\mathcal{A}_A\otimes \mathcal{A}_B$. The tensor product is called \emph{nontrivial} if $\{e\}<A<G$ and $\{e\}<B<G$.

Let $X,Y\in\mathcal{S}$. If $Z\in \mathcal{S}$, then the number of distinct representations of $z\in Z$ in the form $z=xy$ with $x\in X$ and $y\in Y$ does not depend on the choice of $z\in Z$. Denote this number by $c^Z_{XY}$. One can see that $\underline{X}~\underline{Y}=\sum_{Z\in \mathcal{S}(\mathcal{A})}c^Z_{XY}\underline{Z}$. Therefore the numbers  $c^Z_{XY}$ are the structure constants of $\mathcal{A}$ with respect to the basis $\{\underline{X}:\ X\in\mathcal{S}\}$. It can be verified that
\begin{equation}\label{triangle}
|Z|c^{Z^{-1}}_{XY}=|X|c^{X^{-1}}_{YZ}=|Y|c^{Y^{-1}}_{ZX}
\end{equation}
for all $X,Y,Z\in \mathcal{S}(\mathcal{A})$ (see, e.g.,~\cite[Lemma~2.3]{Ry2}).

Let $\mathcal{A}^\prime$ be an $S$-ring over a group $G^\prime$. A bijection $f$ from $\mathcal{S}(\mathcal{A})$ to $\mathcal{S}(\mathcal{A}^\prime)$ is called an \emph{algebraic isomorphism} from $\mathcal{A}$ to $\mathcal{A}^\prime$ if $c_{X^fY^f}^{Z^f}=c_{XY}^Z$ for all $X,Y,Z\in \mathcal{S}(\mathcal{A})$. An algebraic isomorphism from $\mathcal{A}$ to itself is called an \emph{algebraic automorphism} of $\mathcal{A}$. The following statement is known as the first Schur theorem on multipliers (see~\cite[Theorem~23.9, (a)]{Wi}).

\begin{lemm} \label{burn}
Let $\mathcal{A}$ be an $S$-ring over an abelian group  $G$. Then the mapping $X\mapsto X^{(m)}$, $X\in \mathcal{S}(\mathcal{A})$, is an algebraic automorphism of $\mathcal{A}$ for every  $m\in \mathbb{Z}$ coprime to $|G|$.
\end{lemm}

Put $\mathcal{R}(\mathcal{A})=\{r(X):~X\in \mathcal{S}(\mathcal{A})\}$, where $r(X)=\{(g,xg):~x\in X,~g\in G\}$. A bijection $\alpha$ from $G$ to $G^\prime$ is called a \emph{(combinatorial) isomorphism} from $\mathcal{A}$ to $\mathcal{A}^\prime$ if $\mathcal{R}(\mathcal{A})^\alpha=\mathcal{R}(\mathcal{A}^\prime)$, where $\mathcal{R}(\mathcal{A})^\alpha=\{r(X)^\alpha:~X\in \mathcal{S}(\mathcal{A})\}$ and $r(X)^\alpha=\{(g^\alpha,h^\alpha):~(g,h)\in r(X)\}$. If there exists an isomorphism from $\mathcal{A}$ to $\mathcal{A}^\prime$, then $\mathcal{A}$ and $\mathcal{A}^\prime$ are called \emph{isomorphic}.

A bijection $\alpha\in\sym(G)$ is defined to be a \emph{(combinatorial) automorphism} of $\mathcal{A}$ if $r(X)^\alpha=r(X)$ for every $X\in \mathcal{S}(\mathcal{A})$. The set of all automorphisms of $\mathcal{A}$ forms the group called the \emph{automorphism group} of $\mathcal{A}$ and denoted by $\aut(\mathcal{A})$. One can see that $\aut(\mathcal{A})\geq G_r$ and $\mathcal{A}$ is isomorphic to an $S$-ring over a group~$H$ if and only if $\aut(\mathcal{A})$ has a regular subgroup isomorphic to~$H$. If $\alpha\in \aut(\mathcal{A})$, then 
\begin{equation}\label{aut}
(Xy)^\alpha=Xy^\alpha
\end{equation}
for every $X\in \mathcal{S}(\mathcal{A})$ and $y\in G$. If $A$ is an $\mathcal{A}$-subgroup of $G$, then the set of all right $A$-cosets is an imprimitivity system of~$\aut(\mathcal{A})$, and if $\alpha\in \aut(\mathcal{A})$, then $\alpha^{G/A}$ is denoted the permutation induced by $\alpha$ on $G/A$. If $\mathcal{A}=\mathcal{A}_A\otimes \mathcal{A}_B$, then $\aut(\mathcal{A})=\aut(\mathcal{A}_A)\times \aut(\mathcal{A}_B)$.

Let $K$ be a subgroup of $\sym(G)$ containing $G_{r}$. Schur proved in~\cite{Schur} that the $\mathbb{Z}$-submodule
$$V(K,G)=\Span_{\mathbb{Z}}\{\underline{X}:~X\in \orb(K_e,~G)\},$$
where $K_e$ is a stabilizer of $e$ in $K$, is an $S$-ring over $G$. An $S$-ring $\mathcal{A}$ over  $G$ is called \emph{schurian} if $\mathcal{A}=V(K,G)$ for some $K\leq \sym(G)$ with $K\geq G_{r}$. One can verify that $\mathcal{A}$ is schurian if and only if $\mathcal{A}=V(\aut(\mathcal{A}),G)$, or equivalently, $\mathcal{S}(\mathcal{A})=\orb(\aut(\mathcal{A})_e,G)$. Clearly, two isomorphic $S$-rings are schurian or not simultaneously.

Let $K \leq \aut(G)$. Then $\orb(K,G)$ forms a partition of  $G$ that defines the $S$-ring $\mathcal{A}$ over~$G$. In this case  $\mathcal{A}$ is called \emph{cyclotomic} and denoted by $\cyc(K,G)$. If $\mathcal{A}=\cyc(K,G)$ for some $K\leq \aut(G)$, then $\mathcal{A}=V(G_rK,G)$. So every cyclotomic $S$-ring is schurian.

The group $G$ is called a \emph{Schur} group if every $S$-ring over $G$ is schurian. A section (in particular, a subgroup) of a Schur group is also Schur (see, e.g.,~\cite[Theorem~2.4]{PV}).

\section{Nonschurity of $D_8\times \mathbb{Z}_p$}

The main result of this section is the following proposition.

\begin{prop}\label{d8cp}
The group $D_8\times \mathbb{Z}_p$ is not Schur for every prime~$p$. 
\end{prop}

\begin{proof}
The statement of the proposition for $p\in\{2,3\}$ follows from~\cite[Lemma~3.1]{PV}. Further we assume that $p\geq 5$. 

Let $H=\langle a,b:~a^4=b^2=e,a^b=a^{-1} \rangle\cong D_8$, $A_1=\langle a^2 \rangle$, $A_2=\langle ab \rangle$, $C$ a cyclic group of order~$p$, $c$ a generator of~$C$, and $G=H\times C\cong D_8\times \mathbb{Z}_p$. Let us construct a nonschurian $S$-ring over~$G$. Put

$$Z_0=\{e\},~Z_1=\{a^2\},~Z_2=\{ab\},~Z_3=\{a^3b\},~Z_4=\{a,a^3,b,a^2b\},$$
$$X_1=c\{e,a^2\},~X_2=c\{ab,a^3b\},~X_3=c\{a,b\},~X_4=c\{a^3,a^2b\},$$
$$Y_1=c^{-1}\{e,a^2\},~Y_2=c^{-1}\{ab,a^3b\},~Y_3=c^{-1}\{a^3,b\},~Y_4=c^{-1}\{a,a^2b\},$$
$$T_{1k}=c^k\{e,a^2\},~T_{2k}=c^k\{ab,a^3b\},~T_{3k}=c^k\{a,a^3,b,a^2b\},~k\in\{2,\ldots,p-2\}.$$

One can see that the sets $Z_i$, $X_i$, $Y_i$, $T_{jk}$ form the partition of $G$. Denote this partition by $\mathcal{S}$.

\begin{lemm}
The $\mathbb{Z}$-module $\mathcal{A}=\Span_{\mathbb{Z}}\{\underline{X}:~X\in \mathcal{S}\}$ is an $S$-ring over $G$.
\end{lemm}

\begin{proof}
It is easy to check that $Z_i=Z_i^{-1}$ and $Y_i=X_i^{-1}$ for every $i\in\{1,2,3,4\}$ and $T_{j~p-k}=T_{jk}^{-1}$ for every $j\in\{1,2,3\}$ and $k\in\{2,\ldots,p-2\}$. To complete the proof of the lemma, it suffices to check that $\underline{X}\underline{Y}\in \mathcal{A}$ for every $X,Y\in \mathcal{S}$. The definition of $\mathcal{S}$ yields that $X=c^iX_0$ and $Y=c^jY_0$ for some $i,j\in\{0,\ldots,p-1\}$ and $X_0,Y_0\subseteq H$. So $\underline{X}\underline{Y}=c^{i+j}\underline{X_0}\underline{Y_0}$. The straightforward computation in $H\cong D_8$ implies that $\underline{X_0}\underline{Y_0}\in \Span_{\mathbb{Z}}\{\underline{T_0}:~T_0\subseteq H,~c^{i+j}T_0\in \mathcal{S}\}$. Thus, $\underline{X}\underline{Y}\in \mathcal{A}$.
\end{proof}

\begin{lemm}\label{2factor}
In the above notations, the stabilizer $\aut(\mathcal{A}_{G/A_1})_{A_1}$ of $A_1$ in $\aut(\mathcal{A}_{G/A_1})$ is generated by the involution $\sigma\in \sym(G/A_1)$ which interchanges $A_1ac^k$ and $A_1bc^k$ and fixes $A_1c^k$ and $A_1abc^k$ for every $k\in\{0,\ldots,p-1\}$.
\end{lemm}
\begin{proof}
From the definition of $\mathcal{A}$ it follows that $\mathcal{A}_{G/A_1}=\mathcal{A}_{H/A_1}\otimes \mathbb{Z}C$. Therefore $\aut(\mathcal{A}_{G/A_1})=\aut(\mathcal{A}_{H/A_1})\times \aut(\mathbb{Z}C)$. Clearly, $\aut(\mathbb{Z}C)=C_r$. It is easy to see that $\aut(\mathcal{A}_{H/A_1})=(H/A_1)_r\rtimes \langle \sigma^{H/A_1} \rangle\cong \mathbb{Z}_2\wr \mathbb{Z}_2$, where $\sigma^{H/A_1}$ is the restriction of $\sigma$ on $H/A_1$. Thus,
$$\aut(\mathcal{A}_{G/A_1})=((H/A_1)_r\rtimes \langle \sigma^{H/A_1} \rangle)\times C_r.$$
The latter equality implies that $\aut(\mathcal{A}_{G/A_1})_{A_1}=\langle \sigma \rangle$, as required.  
\end{proof}

\begin{lemm}\label{nonschurd8}
The $S$-ring  $\mathcal{A}$ is nonschurian.
\end{lemm}

\begin{proof}
Assume the contrary. Then $\mathcal{S}=\orb(K,G)$, where $K=\aut(\mathcal{A})_e$. Let $\alpha \in K_{ac}$. Observe that $\alpha^{G/A_1}\in \aut(\mathcal{A}_{G/A_1})_{A_1}$. By Lemma~\ref{2factor}, we have $\alpha^{G/A_1}$ is trivial or $\alpha^{G/A_1}=\sigma$. However, the latter case is impossible because $(A_1ac)^{\alpha^{G/A_1}}=A_1ac$. Thus, $\alpha^{G/A_1}$ is trivial.

Since~$\alpha\in K_{ac}$, Eq.~\eqref{aut} implies that $(X_4ac)^\alpha=X_4ac$ and $T_{12}^{\alpha}=T_{12}$. So
$$\{c^2\}^\alpha=(X_4ac\cap T_{12})^\alpha=X_4ac\cap T_{12}=\{c^2\}.$$
By the above equality and Eq.~\eqref{aut}, we have $(X_3c^2)^{\alpha}=X_3c^2$. In addition, $(A_1ac^3)^{\alpha}=A_1ac^3$ because $\alpha^{G/A_1}$ is trivial. So
$$\{ac^3\}^{\alpha}=(X_3ac^3\cap A_1ac^3)^{\alpha}=X_3ac^3\cap A_1ac^3=\{ac^3\}.$$
Thus, $\alpha\in K_{ac^3}$ and hence $K_{ac}\leq K_{ac^3}$.

On the one hand, $|K_{ac}|=|K|/|X_3|=|K|/2$ because $X_3$ is an orbit of~$K$ containing~$ac$. On the other hand, $|K_{ac^3}|=|K|/|T_{33}|=|K|/4$ because $T_{33}$ is an orbit of~$K$ containing~$ac^3$. Therefore $|K_{ac}|=|K|/2>|K|/4=|K_{ac^3}|$, a contradiction to $K_{ac}\leq K_{ac^3}$.   
\end{proof}

Lemma~\ref{nonschurd8} completes the proof of Proposition~\ref{d8cp}.

\end{proof}

\section{Nonschurity of $Q_8\times \mathbb{Z}_p$}

The main result of this section is the following statement. 

\begin{prop}\label{q8cp}
The group $Q_8\times \mathbb{Z}_p$ is not Schur for every prime~$p$ such that $p\equiv1\mod 4$ or $p\equiv1\mod 6$. 
\end{prop}

Let $p$ be a prime and $l$ a divisor of~$p-1$. Suppose that $C\cong \mathbb{Z}_p$. The group $\aut(C)$ is a cyclic group of order~$p-1$. Since $p-1$ is divisible by~$l$, the group $\aut(C)$ has a unique subgroup~$M$ of index~$l$. Put $m=|M|=\frac{p-1}{l}$. It is easy to see that $|\orb(M,C^\#)|=l$ and $\aut(C)$ acts on the set $\orb(M,C^\#)$ as the regular cyclic group of order~$l$. Let $C_1,\ldots,C_l$ be the nontrivial orbits of $M$ on $C$ such that the cycle $f=(C_1 \cdots C_l)\in \sym(\orb(M,C^\#))$ is a generator of $\aut(C)$ acting on $\orb(M,C^\#)$. Clearly, $|C_i|=|C_j|$ for all $i,j$ and $\orb(M,C)$ is the set of the basic sets of the cyclotomic $S$-ring $\mathcal{C}=\cyc(M,C)$. In case $l=4$, we have $C_i=C_i^{-1}$ for every $i$ if $m$ is even and $C_3=C_1^{-1}$, $C_4=C_2^{-1}$ if $m$ is odd. In case $l=6$, we have $C_i=C_i^{-1}$ for every $i$ if $m$ is even and $C_4=C_1^{-1}$, $C_5=C_2^{-1}$, $C_6=C_3^{-1}$ if $m$ is odd.

Put $c_{ij}^k=c_{C_iC_j}^{C_k}$ and $c_{i^f j^f}^{k^f}=c_{C_i^f C_j^f}^{C_k^f}$. From Lemma~\ref{burn} it follows that $f$ is an algebraic automorphism of $\mathcal{C}$ and hence
\begin{equation}\label{constants1}
c_{i^f j^f}^{k^f}=c_{ij}^{k}
\end{equation}
for all $i,j,k\in\{1,\ldots,l\}$.

\begin{lemm}\label{cycnumb4}
In the above notations, let $l=4$.
\\
\\ 
$(1)$ If $m$ is even, then $p=r^2+4s^2$, where $s=c_{12}^{1}-c_{14}^{1}$ and $r$ is an integer. 
\\
\\
$(2)$ If $m$ is odd, then $m=c_{32}^1+c_{34}^1+2c_{41}^1$.
\end{lemm}

\begin{proof}
Observe that $c_{ij}^1=(k-1,j-1)_4$ (see, e.g.~\cite[Eq.~(9)]{BPR}) for all $i,j\in\{1,2,3,4\}$, where $k\in\{1,2,3,4\}$ is such that $C_k=C_i^{-1}$ and $(k-1,j-1)_4$ is a \emph{cyclotomic number} of order~$4$ (see~\cite{Di} for the definition). Therefore Statement~$(1)$ of the lemma follows from~\cite[p.~400,~Eqs.~(50)-(51)]{Di}, whereas Statement~$(2)$ follows from~\cite[p.~401,~Eq.~(54)]{Di}.
\end{proof}

\begin{lemm}\label{cycnumb6}
In the above notations, let $l=6$. 
\\
\\
$(1)$ If $m$ is even, then $4p=r^2+27s^2$, where $s=c_{12}^1+2c_{24}^1+c_{15}^1-c_{13}^1-2c_{25}^1-c_{16}^1$ and $r$ is an integer. Moreover, $3s=t=2u$ or $3s=-t-2u$, where $t=c_{12}^1+2c_{24}^1-3c_{15}^1+3c_{13}^1-c_{16}^1-2c_{25}^1$ and $u=c_{12}^1-c_{16}^1-c_{24}^1+c_{25}^1$.  
\\
\\
$(2)$ If $m$ is odd, then $m=c_{43}^1+c_{45}^1+c_{51}^1+c_{52}^1+2c_{56}^1$. Moreover, $3s=-t=2u$ or $3s=t-2u$, where $s=c_{42}^1+2c_{51}^1+c_{45}^1-c_{43}^1-2c_{52}^1-c_{46}^1$, $t=c_{45}^1+2c_{51}^1-3c_{42}^1+3c_{46}^1-c_{43}^1-2c_{52}^1$, and $u=c_{51}^1-c_{52}^1-c_{43}^1+c_{45}^1$.
\end{lemm}

\begin{proof}
As in the previous lemma, we use the connection between cyclotomic numbers and some intersection numbers that provided by the formula $c_{ij}^1=(k-1,j-1)_4$ (see, e.g.~\cite[Eq.~(9)]{BPR}) for all $i,j\in\{1,2,3,4,5,6\}$, where $k\in\{1,2,3,4,5,6\}$ is such that $C_k=C_i^{-1}$. Now Statement~$(1)$ of the lemma follows from~~\cite[p.~409,~Eqs.~(89)-(90)]{Di} and the relations between the parameters denoted in~\cite[p.~409]{Di} by $M$, $B$ and $F$, whereas Statement~$(2)$ follows from~\cite[p.~410,~Eqs.~(95)-(96)]{Di} and the remark after these equalities.
\end{proof}

Let $H=\langle a,b:~a^4=e,a^2=b^2,a^b=a^{-1} \rangle \cong Q_8$ and $G=H\times C\cong Q_8\times \mathbb{Z}_p$. If $l=4$, then put
$$X_0=\{e\},~X_1=\{a^2\},~X_2=\{ab,a^3b\},~X_3=\{a,a^3,b,a^2b\},$$
$$X_4=C^{\#},~X_5=a^2C^{\#}$$
$$X_6=\{a,a^3\}(C_2\cup C_4)\cup \{b,a^2b\}(C_1\cup C_3),$$
$$X_7=aC_1\cup a^3C_3\cup bC_2\cup a^2bC_4,$$
$$X_8=X_7a^2=aC_3\cup a^3C_1\cup  bC_4\cup a^2bC_2,$$
$$X_9=\{ab,a^3b\}C^{\#}.$$
If $l=6$, then put
$$Y_0=\{e\},~Y_1=\{a^2\},~Y_2=\{a,a^3,b,a^2b,ab,a^3b\},$$
$$Y_3=C^{\#},~Y_4=a^2C^{\#}$$
$$Y_5=\{a,a^3\}(C_1\cup C_4)\cup \{b,a^2b\}(C_2\cup C_5)\cup \{ab,a^3b\}(C_3\cup C_6),$$
$$Y_6=\{a,a^3\}(C_2\cup C_5)\cup \{b,a^2b\}(C_3\cup C_6)\cup \{ab,a^3b\}(C_1\cup C_4),$$
$$Y_7=aC_3\cup a^2bC_4 \cup abC_5 \cup a^3C_6 \cup bC_1 \cup a^3bC_2,$$
$$Y_8=Y_7a^2=aC_6\cup a^2bC_1 \cup abC_2 \cup a^3C_3 \cup bC_4 \cup a^3bC_5.$$

The sets $X_i$, $i\in I=\{0,\ldots,9\}$, and $Y_j$, $j\in J=\{0,\ldots,8\}$, form the partitions of $G$. Denote these partition by $\mathcal{S}$ and $\mathcal{T}$, respectively. Put $\mathcal{A}=\Span_{\mathbb{Z}}\{\underline{X}:~X\in \mathcal{S}\}$ and $\mathcal{B}=\Span_{\mathbb{Z}}\{\underline{Y}:~Y\in \mathcal{T}\}$.

\begin{lemm}
The $\mathbb{Z}$-modules $\mathcal{A}$ and $\mathcal{B}$ are $S$-rings over $G$.
\end{lemm}

\begin{proof}
One can see that $X_i=X_i^{-1}$ for $i\in I\setminus \{7,8\}$ and $Y_j=Y_j^{-1}$ for $j\in J\setminus \{7,8\}$. If $m$ is even, then $X_8=X_7^{-1}$ and $Y_8=Y_7^{-1}$; for otherwise $X_7=X_7^{-1}$, $X_8=X_8^{-1}$, $Y_7=Y_7^{-1}$, and $Y_8=Y_8^{-1}$.

Let $\sigma_1,\sigma_2,\sigma_3\in \aut(H)$ such that $\sigma_1:(a,b)\mapsto (b,a^3)$, $\sigma_2:(a,b)\mapsto (a^3,b)$, and $\sigma_3:(a,b)\mapsto (a^2b,a^3b)$.  Put 
$$U=\langle \sigma_1,\sigma_2 \rangle,~V=\langle \sigma_1^2,\sigma_2,\sigma_3 \rangle,~\text{and}~U_0=V_0=\langle \sigma_1^2, \sigma_2 \rangle.$$
One can verify straightforwardly that $U\cong D_8$, $V\cong A_4$, and $U_0=V_0\cong \mathbb{Z}_2\times \mathbb{Z}_2$. Moreover, $U_0=V_0$ is normal in the both $U$ and $V$. Put also $W=\aut(C)$. Denote by~$W_0$ the unique subgroup of~$W$ such that $W_0>M$ and $|W_0:M|=2$ (such group $W_0$ exists because $l$ is even). The canonical epimorphisms from $U$ to $U/U_0$, from $V$ to $V/V_0$, and from $W$ to $W/W_0$ are denoted by $\pi_1$, $\pi_2$, and $\pi_3$, respectively.  

If $l=4$, then there exists the unique isomorphism $\psi$ from $U/U_0\cong \mathbb{Z}_2$ to $W/W_0\cong \mathbb{Z}_2$. Put
$$K_1=\{(\sigma,\tau)\in U\times W:~(\sigma^{\pi_1})^\psi=\tau^{\pi_3}\}\leq \aut(G).$$
If $l=6$, then there exists the unique isomorphism $\theta$ from $V/V_0\cong \mathbb{Z}_3$ to $W/W_0\cong \mathbb{Z}_3$ such that $(V_0\sigma_3)^\theta=W_0^{f^{\prime}}$, where $f^{\prime}\in W$ acting on $\orb(M,C^\#)$ induces $f$. Put 
$$K_2=\{(\sigma,\tau)\in V\times W:~(\sigma^{\pi_2})^\theta=\tau^{\pi_3}\}\leq \aut(G).$$
The groups $K_1$ and $K_2$ are subdirect products of $U$ and $W$ and $V$ and $W$, respectively. The straightforward check shows that the basic sets of the cyclotomic $S$-ring $\cyc(K_1,G)$ are $X_i$, $i\in I\setminus \{7,8\}$, and $X_7\cup X_8$, whereas the basic sets of the cyclotomic $S$-ring $\cyc(K_2,G)$ are $Y_j$, $j\in J\setminus \{7,8\}$, and $Y_7\cup Y_8$. This implies that $\underline{X_iX_j}\in \mathcal{A}$ for every $i,j\in I\setminus \{7,8\}$ and $\underline{Y_iY_j}\in \mathcal{B}$ for every $i,j\in J\setminus \{7,8\}$. Thus, to prove that $\mathcal{A}$ and $\mathcal{B}$ are $S$-rings, it remains to verify that $\underline{X_iX_j}\in \mathcal{A}$ and $\underline{Y_iY_j}\in \mathcal{B}$, where at least one of $i$, $j$ belongs to the set $\{7,8\}$. The latter can be done by the straightforward computation using Eq.~\eqref{constants1}.
\end{proof}

\begin{lemm}\label{nonschurq81}
The $S$-ring  $\mathcal{A}$ is not schurian.
\end{lemm}

\begin{proof}
Assume the contrary. Then $\mathcal{S}=\orb(K,G)$, where $K=\aut(\mathcal{A})_e$. Let $c_2\in C_2$ and $c_4\in C_4$. The elements $ac_2$ and $ac_4$ belong to $X_6$. Since $X_6$ is an orbit of~$K$, there exists  $\alpha\in K$ such that $(ac_2)^{\alpha}=ac_4$. Due to Eq.~\eqref{aut}, we have 
\begin{equation}\label{c1c}
(C_1c_2)^{\alpha}=(X_8ac_2\cap X_4)^{\alpha}=X_8ac_4\cap X_4=C_1c_4.
\end{equation}

The set $Ca$ is a block of $K$ because $C$ is an $\mathcal{A}$-subgroup. Observe that $(ac_2)^{\alpha}=ac_4\in(Ca)^\alpha\cap Ca$. So $(Ca)^\alpha=Ca$.  Therefore
$$\{a\}^{\alpha}=(Ca \cap X_3)^{\alpha}=Ca\cap X_3=\{a\}.$$
The above equality and Eq.~\eqref{aut} yield that
\begin{equation}\label{c1}
C_1^{\alpha}=(X_8a\cap X_4)^{\alpha}=X_8a\cap X_4=C_1
\end{equation}
and
\begin{equation}\label{c3}
C_3^{\alpha}=(X_7a\cap X_4)^{\alpha}=X_7a\cap X_4=C_3.
\end{equation}

Suppose that $m$ is even. Then $c_{11}^2=|C_1c_2\cap C_1|$ and $c_{11}^4=|C_1c_4\cap C_1|$. Eqs.~\eqref{c1c} and~\eqref{c1} imply that 
$$c_{11}^2=|C_1c_2\cap C_1|=|(C_1c_2\cap C_1)^{\alpha}|=|C_1c_4\cap C_1|=c_{11}^4.$$
By the latter equality and Eq.~\eqref{triangle}, we have $c_{12}^{1}=c_{14}^{1}$. Therefore $p=r^2$ for some integer $r$ by Lemma~\ref{cycnumb4}(1), a contradiction to primality of~$p$. 

Now suppose that $m$ is odd. Then $c_{33}^2=|C_1c_2\cap C_3|$ and $c_{33}^4=|C_1c_4\cap C_3|$. From Eqs.~\eqref{c1c} and~\eqref{c3} it follows that 
$$c_{33}^2=|C_1c_2\cap C_3|=|(C_1c_2\cap C_3)^{\alpha}|=|C_1c_4\cap C_3|=c_{33}^4.$$
Due to the latter equality and Eq.~\eqref{triangle}, we have $c_{34}^{1}=c_{32}^{1}$. Therefore $m=2c_{32}^1+2c_{41}^1$ by Lemma~\ref{cycnumb4}(2), a contradiction to the oddity of~$m$.
\end{proof}

\begin{lemm}\label{nonschurq82}
The $S$-ring  $\mathcal{B}$ is not schurian.
\end{lemm}

\begin{proof}
Assume the contrary. Then $\mathcal{T}=\orb(L,G)$, where $L=\aut(\mathcal{B})_e$. Let $c_1\in C_1$. The elements $ac_1$ and $a^3c_1$ belong to $Y_5$ and the elements $abc_1$ and $a^3bc_1$ belong to $Y_6$. Since $Y_5$ and $Y_6$ are orbits of~$L$, there exist  $\beta,\gamma\in L$ such that $(ac_1)^{\beta}=a^3c_1$ and $(abc_1)^{\gamma}=a^3bc_1$. Eq.~\eqref{aut} implies that 
\begin{equation}\label{aa3}
(C_3c_1)^{\beta}=(Y_8ac_1\cap Y_3)^{\beta}=Y_8a^3c_1\cap Y_3=C_6c_1
\end{equation}
and 
\begin{equation}\label{aba3b}
(C_2c_1)^{\gamma}=(Y_7abc_1\cap Y_3)^{\gamma}=Y_7a^3bc_1\cap Y_3=C_5c_1.
\end{equation}

Since $C$ is a $\mathcal{B}$-subgroup, the sets $Ca$,$Ca^3$, $Cab$, and $Ca^3b$ are blocks of~$L$. Together with $(ac_1)^{\beta}=a^3c_1$ and $(abc_1)^{\gamma}=a^3bc_1$, this yields that $(Ca)^{\beta}=Ca^3$ and $(Cab)^{\gamma}=Ca^3b$. Therefore by Eq.~\eqref{aut}, we have
$$\{a\}^{\beta}=(Ca \cap Y_2)^{\beta}=Ca^3\cap Y_2=\{a^3\}$$
and 
$$\{ab\}^{\gamma}=(Cab \cap Y_2)^{\gamma}=Ca^3b\cap Y_2=\{a^3b\}$$
Due to the above equalities and Eq.~\eqref{aut}, we obtain
\begin{equation}\label{c36}
C_3^{\beta}=(Y_8a\cap Y_3)^{\beta}=Y_8a^3\cap Y_3=C_6,
\end{equation}
\begin{equation}\label{c63}
C_6^{\beta}=(Y_7a\cap Y_3)^{\beta}=Y_7a^3\cap Y_3=C_3,
\end{equation}
and 
\begin{equation}\label{c25}
C_2^{\gamma}=(Y_7ab\cap Y_3)^{\gamma}=Y_7a^3b\cap Y_3=C_5,
\end{equation}
\begin{equation}\label{c52}
C_5^{\gamma}=(Y_8ab\cap Y_3)^{\gamma}=Y_8a^3b\cap Y_3=C_2.
\end{equation}

Suppose that $m$ is even. Then $c_{33}^1=|C_3c_1\cap C_3|$ and $c_{66}^1=|C_6c_1\cap C_6|$. Eqs.~\eqref{aa3} and~\eqref{c36} imply that $c_{33}^1=c_{66}^1$. By the latter equality and Eq.~\eqref{constants1}, we have $c_{11}^{2}=c_{11}^{5}$. Therefore 
\begin{equation}\label{c22c55}
c_{12}^1=c_{15}^1
\end{equation} 
by Eq.~\eqref{triangle}. Since $m$ is even, $c_{22}^1=|C_2c_1\cap C_2|$ and $c_{55}^1=|C_5c_1\cap C_5|$. Eqs.~\eqref{aba3b} and~\eqref{c25} imply that $c_{22}^1=c_{55}^1$. By the latter equality and Eq.~\eqref{constants1}, we have $c_{11}^{3}=c_{11}^{6}$. Therefore 
\begin{equation}\label{c33c66}
c_{13}^1=c_{16}^1
\end{equation} 
by Eq.~\eqref{triangle}. 

The definitions of $t$ and $u$ from Lemma~\ref{cycnumb6}(1) and Eqs.~\eqref{c22c55} and~\eqref{c33c66} yield that $t=-2u$. Together with each of the equalities $3s=t=2u$ and $3s=-t-2u$ (see the second part of Lemma~\ref{cycnumb6}(1)), the latter implies that $s=0$. Therefore $4p=r^2$ for some integer $r$ by the first part of Lemma~\ref{cycnumb6}(1), a contradiction to primality of~$p$.

Now suppose that $m$ is odd. Then $c_{33}^1=|C_6c_1\cap C_3|$ and $c_{66}^1=|C_3c_1\cap C_6|$. Due to Eqs.~\eqref{aa3} and~\eqref{c63}, we obtain $c_{33}^1=c_{66}^1$. By the latter equality and Eq.~\eqref{constants1}, we have $c_{44}^{2}=c_{44}^{5}$. Therefore 
\begin{equation}\label{c42c45}
c_{42}^1=c_{45}^1
\end{equation} 
by Eq.~\eqref{triangle}. Since $m$ is odd, $c_{22}^1=|C_5c_1\cap C_2|$ and $c_{55}^1=|C_2c_1\cap C_5|$. Eqs.~\eqref{aba3b} and~\eqref{c52} imply that $c_{22}^1=c_{55}^1$. By the latter equality and Eq.~\eqref{constants1}, we have $c_{44}^{3}=c_{44}^{6}$. Therefore 
\begin{equation}\label{c43c46}
c_{43}^1=c_{46}^1
\end{equation} 
by Eq.~\eqref{triangle}. 

Due to the definitions of $t$ and $u$ from Lemma~\ref{cycnumb6}(2) and Eqs.~\eqref{c42c45} and~\eqref{c43c46}, we obtain $t=2u$. Together with each of the equalities $3s=-t=2u$ or $3s=t-2u$ (see the second part of Lemma~\ref{cycnumb6}(2)), the latter implies that $s=0$. Therefore 
$$s=c_{42}^1+2c_{51}^1+c_{45}^1-c_{43}^1-2c_{52}^1-c_{46}^1=2c_{45}^1+2c_{51}^1-2c_{43}^1-2c_{52}^1=0,$$
where the second equality holds by Eqs.~\eqref{c42c45} and~\eqref{c43c46}. So $c_{45}^1+c_{52}^1=c_{43}^1+c_{52}^1$. Now from the first part of Lemma~\ref{cycnumb6}(2) it follows that $m=c_{43}^1+c_{45}^1+c_{51}^1+c_{52}^1+2c_{56}^1=2(c_{43}^1+c_{52}^1)+2c_{56}^1$, a contradiction to the oddity of~$m$. 
\end{proof}

Proposition~\ref{q8cp} follows from Lemma~\ref{nonschurq81} if $p\equiv1\mod 4$ and from Lemma~\ref{nonschurq82} if $p\equiv1\mod 6$.

\section{Proof of Theorem~\ref{main}}

We start with a lemma on the structure of nilpotent Schur groups. 

\begin{lemm}\label{necessary}
Let $G$ be a nonabelian nilpotent Schur group. Then $G=P\times H$, where $P$ is isomorphic to one of the groups $G_{16}$, $Q_8$, $D_{2^k}$, $k\geq 3$, and $H$ is a cyclic group of odd order. 
\end{lemm}

\begin{proof}
Since $G$ is nilpotent, it is a direct product of its Sylow subgroups. At least one of the Sylow subgroups of~$G$, say Sylow $p$-subgroup $P$, is nonabelian, because $G$ is nonabelian. The group $P$ is Schur as a subgroup of the Schur group~$G$. From~\cite[Theorem~4.2]{PV}, \cite[Theorem~1.2]{MP2}, and~\cite[Theorem~1]{Ry1} it follows that $p=2$ and $P$ is isomorphic to one of the groups $G_{16}$, $Q_8$, $D_{2^k}$, $k\geq 3$. Let $H$ be the Hall $2^{\prime}$-subgroup of $G$. Then $G=P\times H$. Due to~\cite[Corollary~4.3]{PV}, the group $H$ is cyclic.
\end{proof}

We continue with one more lemma on the groups $Q_8\times \mathbb{Z}_{pq}$.
 
\begin{lemm}\label{q8cpq}
The group $Q_8\times \mathbb{Z}_{pq}$ is not Schur for every odd primes~$p$ and~$q$. 
\end{lemm}

\begin{proof}
Let $A\cong \mathbb{Z}_4$, $B\cong \mathbb{Z}_2$, and $C\cong \mathbb{Z}_{pq}$. Let $a$, $b$, and $c$ be generators of $A$, $B$, and $C$, respectively. Put $a_1=a^2$, $A_1=\langle a_1\rangle$, $c_1=c^q$, $C_1=\langle c_1 \rangle\cong C_p$, and $G=A\times B\times C\cong \mathbb{Z}_4 \times \mathbb{Z}_2\times \mathbb{Z}_{pq}$. To prove Proposition~\ref{q8cpq}, we will show that the nonschurian $S$-ring $\mathcal{A}$ over $G$ constructed in~\cite[Theorem~4.1]{EKP2} is isomorphic to an $S$-ring over $Q_8\times \mathbb{Z}_{pq}$. 

Let us recall that $\mathcal{A}$ in the $(A\times C)/A_1$-wreath product in our notations (see~\cite[pp.~8-10]{EKP2}). Therefore $\sigma\in \aut(G)$ such that $\sigma:(a,b,c)\mapsto (a,a_1b,c)$ is an automorphism of~$\mathcal{A}$. Since $G_r\leq \aut(\mathcal{A})$, we conclude that $\sigma a_r,\sigma b_r\in \aut(\mathcal{A})$. The straightforward computation implies that $|\sigma a_r|=|\sigma b_r|=4$, $(\sigma a_r)^2=(\sigma b_r)^2$, and $(\sigma a_r)^{\sigma b_r}=(\sigma a_r)^{-1}$. So $H=\langle \sigma a_r,\sigma b_r\rangle \cong Q_8$. Moreover, one  can easily verify that $H$ is transitive and hence regular on $A\times B$. Thus, $H\times C_r$ is a regular subgroup of $\aut(\mathcal{A})$ isomorphic to~$Q_8\times \mathbb{Z}_{pq}$. The latter yields that $\mathcal{A}$ is isomorphic to an $S$-ring over~$Q_8\times \mathbb{Z}_{pq}$. Since $\mathcal{A}$ is nonschurian by~\cite[Theorem~4.1]{EKP2}, the group $Q_8\times \mathbb{Z}_{pq}$ is not Schur. 
\end{proof}

\begin{proof}[Proof of Theorem~\ref{main}]

Schurity of the groups $Q_8$, $G_{16}$, and $D_{2^k}$, where $3\leq k\leq 5$, follows from~\cite[Theorem~4.2]{PV}. Let $G$ be a nonabelian nilpotent Schur group. Then $G=P\times H$, where $P$ is isomorphic to one of the groups $G_{16}$, $Q_8$, $D_{2^k}$, $k\geq 3$, and $H$ is a cyclic group of odd order be Lemma~\ref{necessary}. We may assume that $H$ is nontrivial because for otherwise the statement of the theorem holds. If $P$ is isomorphic to $D_{2^k}$, $k\geq 3$, or $G_{16}$, then $H$ has a subgroup isomorphic to~$D_8$ and hence $G$ has a subgroup isomorphic to $D_8\times \mathbb{Z}_p$ for some odd prime~$p$. The latter group is not Schur by Proposition~\ref{d8cp}. Therefore $G$ is not Schur, a contradiction to the assumption.

The above paragraph yields that $P$ is isomorphic to~$Q_8$. If $|H|$ has at least two prime divisors (possibly, equal), then $H$ has a subgroup isomorphic to $\mathbb{Z}_{pq}$ for some odd primes $p$ and $q$ and hence $G$ has a subgroup isomorphic to $Q_8\times \mathbb{Z}_{pq}$. The latter group is not Schur by Lemma~\ref{q8cpq}. Therefore $G$ is not Schur, a contradiction to the assumption. Thus, $H$ is a cyclic group of odd prime order and we are done by~\cite[Lemma~3.1]{PV} if $p\leq 3$ and by Proposition~\ref{q8cp} otherwise.
\end{proof}

\end{document}